\documentclass[12pt, reqno]{amsart}

\usepackage{amssymb, amsmath, amsfonts, amscd, calligra, mathrsfs}
\usepackage[raiselinks=false,colorlinks=true,citecolor=blue,urlcolor=blue,
linkcolor=blue,bookmarksopen=true,pdftex]{hyperref}
\usepackage[dvipsnames]{xcolor}
\usepackage{tikz}\usepackage{tikz-cd}
\usepackage{enumerate}
\usepackage[mathscr]{eucal}
\newtheorem{theorem}{Theorem}
\newtheorem{proposition}[theorem]{Proposition}
\newtheorem{lemma}[theorem]{Lemma}
\newtheorem{remark}[theorem]{Remark}

\newtheorem{example}[theorem]{Example}
\newtheorem{corollary}[theorem]{Corollary}

\renewcommand{\hom}{\textrm{Hom}}

\newcommand{\Pic}{\textrm{Pic}}
\newcommand{\rank}{\textrm{rank}}
\begin{document}
\baselineskip=15.5pt
\title[Picard groups of certain compact complex parallelizable manifolds]{Picard groups of certain compact complex parallelizable manifolds and related spaces}
\author[P. Biswas]{Pritthijit Biswas} 
\address{Chennai Mathematical Institute, H1 Sipcot IT Park, Siruseri, Kelambakkam 603103.}
\email{pritthijit@cmi.ac.in}
\author[P. Sankaran]{Parameswaran Sankaran
}
\email{sankaran@cmi.ac.in}
\subjclass[2010]{32M10; 32Q99}
\keywords{Picard group, complex semisimple Lie group, non-K\"ahler compact complex 
manifolds}
\thispagestyle{empty}
\date{}

\begin{abstract}
Let $G$ be a complex simply connected semisimple Lie group and let $\Gamma$ be a torsionless uniform irreducible lattice in $G$.  
Then $\Gamma\backslash G$ is a compact complex non-K\"ahler manifold whose tangent bundle is holomorphically 
trivial. In this note we compute the Picard group of $\Gamma\backslash G$ when $\rank(G)\geq 3$. 
When $\rank(G)\lneq 3$, we determine the group $Pic^0(\Gamma\backslash G)\subset Pic(\Gamma\backslash G)$ of topologically trivial holomorphic 
line bundles. 
 
When $\rank(G)\ge 2$, we also show that $Pic^0(P_\Gamma)$ is isomorphic to $Pic^0(Y)$ where $P_\Gamma$ is a 
$\Gamma\backslash G$-bundle associated to a principal $G$-bundle over a compact connected complex 
manifold $Y$, and, when $\rank(G)\ge 3$, we show that $Pic(Y)\to Pic(P_\Gamma)$ is injective 
with finite cokernel.  
\end{abstract}

\maketitle

\section{Introduction}\label{intro}

Let $G$ be a simply connected complex linear Lie group and $\Gamma\subset G$ be a uniform lattice in $G$ and let 
$M_\Gamma=\Gamma\backslash G$. 
Then $M_\Gamma$ is a compact homogeneous manifold which has the property that its tangent 
bundle is holomorphically trivial.   Special cases of such manifolds include complex tori.  A result of 
H.-C. Wang \cite[Corollary 2]{W} shows that 
when $G$ is non-abelian, $M_\Gamma$ is not a K\"ahler manifold. See also \cite{BR}.

Denote by $\Pic(X)$ the group of isomorphism classes of holomorphic line bundles over a connected complex manifold $X$ and by $\Pic^0(X)$ 
the subgroup of topologically trivial line bundles. Thus $\Pic(X)\cong H^1(X,\mathcal O^*_X)$ and $\Pic^0(X)$ is the kernel 
of the Chern class map $c_1: \Pic(X)\to H^2(X;\mathbb Z).$  When $X$ is compact and K\"ahler, $Pic^0(X)$ is a complex torus (diffeomorphic 
to a product of the circle group). 
 When $\Gamma \subset V$ is a lattice in an $n$-dimensional complex vector space,  
$X:=\Gamma\backslash V$ is a complex torus, 
$\Pic^0(X)$ is again an $n$-dimensional complex torus, namely the dual torus $X^\vee=V^\vee/\Gamma^\vee$ 
where 
$V^\vee $ is the space of conjugate linear functions $V\to \mathbb C$ and $ \Gamma^\vee=\{f\in V^\vee\mid \mathfrak{Im} f(\Gamma)\subset \mathbb Z\}$.   See \cite{mumford}. 
 
In this note, we will be concerned with the determination of the Picard group when $M_\Gamma=\Gamma\backslash G$ 
where $G$ is a simply connected complex semisimple Lie group and $\Gamma$ is a torsionless irreducible uniform lattice 
in $G$.  
We shall also consider compact connected complex manifolds which 
are holomorphically fibred by $M_\Gamma$. 
 Such holomorphic bundles arise naturally. Indeed, if $\pi: E\to B$ is a holomorphic principal 
$G$-bundle over $B$, where $G$ acts on the left of $E$, then $\Gamma\backslash E\to B$ is such an $M_\Gamma$-bundle over $B$.

Recall that the real rank of a linear connected real semisimple Lie group $G$, denoted $\rank_\mathbb R(G)$ is the maximum dimension  $r$ of a diagonalizable subgroup of $G$ isomorphic to $(\mathbb R^\times_{>0})^r $.    This definition is applicable even if $G$ is a complex semisimple Lie group, by regarding it as a real Lie group, ignoring the 
complex structure.  In this case, the real rank of $G$ equals the rank of $G$ viewed as a complex 
algebraic group, namely, the dimension of a 
maximal (algebraic) torus (isomorphic to $(\mathbb C^*)^r$) contained in $G$.  For example, 
when $G=SL(n,\mathbb C)$, we have $\rank_\mathbb RG=n-1$.  In view of this, we shall henceforth write $\rank(G)$ to mean $\rank_\mathbb R(G)$.   Any simply connected semisimple complex Lie group is a direct product of  simple Lie groups. Here, as is customary in the theory of Lie groups, a connected Lie group is {\it simple} if its only proper normal subgroups are discrete, which are necessarily contained in the centre of the Lie group.  The centre of any complex semisimple Lie group is finite.   
We refer the reader to \cite{V} for these and other basic facts 
about complex semisimple Lie groups, including their classification.   

A discrete subgroup $\Gamma$ in a semisimple Lie group $G$ with finitely many connected components is called a {\it lattice} if $M_\Gamma=\Gamma\backslash G$ 
has a $G$-invariant measure (for the right action of $G$ on $M_\Gamma$) with respect to which 
it has finite volume.\footnote{Some authors assume that $G$ has no compact factors (i.e. $G$ has no proper normal compact subgroups of positive dimension) while defining a lattice.}  If $M_\Gamma$ is compact, then $\Gamma$ is said to be {\it uniform} (or {\it cocompact}). By a well-known result of Borel and Harish-Chandra, any (real) semisimple Lie group with finitely many connected components has both uniform as well as non-uniform lattices. A lattice $\Gamma$ in $G$ is 
said to be {\it reducible} if there exists a finite cover $\pi:\tilde G\to G$ such that $\tilde G=\tilde G_0\times \tilde G_1$ 
where $\tilde G_i$ are non-compact subgroups of $\tilde G$ and a subgroup $\Gamma_0\subset \Gamma$ having finite index in $\Gamma$ such that $\pi^{-1}(\Gamma_0)$ is a product $\Lambda_0\times \Lambda_1$ where $\Lambda_j\subset \tilde G_j$ is a lattice for $j=0,1$. If $\Gamma$ is 
not reducible, then it is said to be {\it irreducible}. For example, any lattice in a simple Lie group is irreducible.
 The reader is referred to the book by Raghunathan \cite{raghunathan} for the definition of lattices in a general Lie group and standard facts concerning them.

\begin{theorem} \label{picardgroup}  Let $G$ be a simply connected complex semisimple Lie group and let $\Gamma$ be a torsionless irreducible uniform lattice in $G$.   Let $M_\Gamma:=\Gamma\backslash G$. Then: \\
(i) If $\rank(G)\ge 2$, then $Pic^0(M_\Gamma)=0$.  If $\rank(G)\ge 3$, then $Pic(M_{\Gamma}) \cong H^2(M_\Gamma;\mathbb Z)$ is a finite group. \\
(ii) If $\rank(G)=1$, then $Pic^0(M_\Gamma)\cong \mathbb C^r/\mathbb Z^r$, where $r=\rank (\Gamma/[\Gamma,\Gamma])$.

\end{theorem}

The value of $r$ in the rank $1$ case, namely when $G=SL(2,\mathbb C)$ can be arbitrarily large. See Remark 6(ii) below. 

We have not been able to determine the image of the Chern class map $Pic(M_\Gamma)\to H^2(M_\Gamma;\mathbb Z)$ in the case when $\rank (G)\le 2$.  
By the well-known classification theorem,  the simply connected complex semisimple Lie groups of rank $2$ are 
$SL(3,\mathbb C), SL(2,\mathbb C)\times SL(2,\mathbb C), Spin(5)$ and the exceptional Lie group $G_2$. When $\rank(G)=1$, $G$ is isomorphic to $SL(2,\mathbb C)$.  

Suppose that $Y$ is a connected compact complex manifold and $\pi:P\to Y$ is a principal holomorphic fibre bundle with 
fibre and structure group $G$ where $G$ is  a simply connected complex semisimple Lie group.   We regard $G$ as 
acting on the left of $P$.  Let $P_\Gamma\to Y$ be the associated $M_\Gamma=\Gamma\backslash G$-bundle where 
$\Gamma$ is an irreducible torsionless lattice in $G$.   Thus $P_\Gamma=\Gamma\backslash P$. Since $M_\Gamma$ is not a K\" ahler manifold, neither is $P_\Gamma$.

\begin{theorem}  \label{bundleversionofPic}  
Let $\Gamma, G$ be as in Theorem \ref{picardgroup} and suppose that $\textrm{rank}(G)\ge 2$.   With the above notations, $\pi:P_\Gamma \to Y$ induces an isomorphism $\pi^*: \Pic^0(P_\Gamma)\cong \Pic^0(Y)$.  Moreover, 
if $\rank(G)\ge 3$,  then $\pi^*:Pic(Y)\to Pic(P_\Gamma)$ is injective and has finite cokernel.  
\end{theorem}

 Elliptic curve bundles with non-K\"ahler total spaces first appeared in the work of H. Hopf \cite {hopf} who showed that $\mathbb S^1\times \mathbb S^{2n-1}$ admits a complex structre. Since the work of 
Calabi and Eckmann \cite{ce} who obtained elliptic bundle structrures 
on product of two odd dimensional spheres, many researchers have constructed new classes of compact non-K\"ahler  complex manifolds and studied their geometry.    We point out but only a few:  A. Blanchard, \cite{blanchard}, C. Borcea \cite{borcea}, T. H\"ofer \cite{hofer},  V. Br\^{i}nz\u{a}nescu \cite{brinzanescu}, V. Ramani and P. Sankaran \cite{ramani-sankaran},  
S. L\'opez de Madrano and A. Verjovsky \cite{lv}, 
J. J. Loeb and M. Nicolau \cite{ln}, L. Meersseman \cite{meersseman},  Meersseman and Verjovsky \cite{mv}, 
Sankaran and A. S. Thakur \cite{st}, and M. Poddar and Thakur \cite{poddar-thakur}.   The reader is referred to 
the memoir  by J. Winkelmann \cite{Wi} for a study of the geometry of complex parallelizable manifolds.

Proofs of both Theorems \ref{picardgroup} and \ref{bundleversionofPic} uses the exponential exact sequence and 
the description of the $\bar\partial$-cohomology of $\Gamma\backslash G$ as a $G$-module due to Akhiezer \cite{akhiezer2}, and a criterion for the vanishing  
of $H^q(\Gamma;\mathbb C)$ (\cite[\S4, Ch. VII]{borel-wallach}) when $q\le 2$.  (These results 
will be recalled in the sequel.)  In addtion, proof of Theorem \ref{bundleversionofPic} 
uses the Borel spectral sequence \cite[Appendix II]{hirzebruch}, which we shall recall in \S \ref{borelss}. 
Theorem \ref{picardgroup} will be proved in \S2 and  Theorem \ref{bundleversionofPic},  
in \S 3.

\section{Cohomology of $\Gamma\backslash G$}
Let $G$ be a connected real semisimple linear Lie group with finite centre and with no compact factors.  Let $r=\rank(G)$.
 
Let $K\subset G$ be a maximal compact subgroup of $G$.  Then $X= G/K$ is a globally symmetric space diffeomorphic 
to a cell.   Let $\Gamma$ be a uniform lattice in $G$ which 
is torsionless.   
Then $\Gamma 
$ acts on $G/K$ freely and properly discontinuously 
and the locally symmetric space $X_\Gamma:=\Gamma\backslash G/K$ is a smooth connected 
compact manifold which is an Eilenberg-MacLane complex $K(\Gamma,1)$.  
Denote by $\mathfrak g$ (resp. $\mathfrak k$) the Lie algebra of $G$ (resp. $K$).   One has the Cartan decomposition $\mathfrak g=\mathfrak k\oplus \mathfrak p$ where $\mathfrak p=\mathfrak k^\perp$, the orthogonal complement with respect to the Killing form on $\mathfrak g$.  Denote by $\mathfrak p_\mathbb C$ the 
complexification $\mathfrak p\otimes \mathbb C$ of the vector space $\mathfrak p$.  Then $\mathfrak p_\mathbb C$ is a complex representation space for $K$ (given by the adjoint action).  

Let $C^\infty(\Gamma\backslash G)$ denote the complex vector space of smooth complex valued functions on $\Gamma\backslash G$ and let 
$L^2(\Gamma\backslash G)$ denote the complex Hilbert space of square integrable functions on $\Gamma\backslash G$ with respect to the $G$-invariant measure obtained from the Haar measure on $G$.  The translation action 
(on the right) of $G$ on $\Gamma\backslash G$ induces an action of $G$ on $L^2(\Gamma\backslash G)$ making 
the latter a representation space.  Since the measure on $\Gamma\backslash G$ is $G$-invariant, the $G$ representation on $L^2(\Gamma\backslash G)$ is unitary. 
 
The cohomology $H^*(\Gamma;\mathbb C):=H^*(X_\Gamma;\mathbb C)$ has been described by Matsushima 
\cite{matsushima} 
as the relative Lie algebra cohomology $H^*(\mathfrak g,  K;C^\infty (\Gamma\backslash G)_{(K)})$ 
when $\Gamma$ is a uniform lattice.   See also 
\cite{mm}, \cite[Chapter VII]{borel-wallach}.
Here $C^\infty(\Gamma\backslash G)_{(K)}\subset C^\infty(\Gamma\backslash G)\subset L^2(\Gamma\backslash G)$ denotes the $(\mathfrak g,K)$-module of smooth $K$-finite vectors of  $L^2(\Gamma\backslash G)$.  More precisely, 
by a theorem of Gelfand and Piatetsky-Shapiro \cite[p.23]{ggp}, the 
Hilbert space $L^2(\Gamma\backslash G)$ decomposes as a Hilbert direct sum  $\hat{ \oplus} m(\Gamma,\pi) H_\pi$ of certain irreducible unitary 
representations $(\pi,H_\pi)$ of $G$ occurring with {\it finite} multiplicity $m(\Gamma,\pi)$.   We denote by $V_{\sigma ,(K)}$ the space of all smooth $K$-finite vectors of a unitary representation $(\sigma,V_\sigma)$ of $G$ on a Hilbert space.   
Then $V_{\sigma ,(K)}$ is a $(\mathfrak g,K)$-module.  
We are ready to state the result of 
Matsushima.

\begin{theorem}\label{matsushima} {\em (\cite{matsushima},\cite{mm})}With the above notations, the cohomology
of the compact locally symmetric space $X_\Gamma=\Gamma\backslash G/K$ has the following description. 
\begin{equation}
H^r(X_\Gamma;\mathbb C)\cong H^r(\Gamma;\mathbb C)\cong H^r(\mathfrak g,K;L^2(\Gamma\backslash G)_{(K)})\cong \bigoplus m(\Gamma,\pi) H^r(\mathfrak g,K;H_{\pi,(K)})
\end{equation}
where the sum is over all isomorphism classes of irreducible unitary representations $\pi$ of $G$ that occur in $L^2(\Gamma\backslash G)$ 
with  positive multiplicity $m(\Gamma,\pi)$.   
\end{theorem}

Note that, since the left most side of (1) is a finite dimensional vector space, there are only finitely many representations $\pi$ with $m(\Gamma,\pi)>0$ with non-vanishing $(\mathfrak g, K)$-cohomology.    Since $\Gamma \backslash G$ has finite volume, the space $L^2(\Gamma\backslash G)$ contains the 
trivial representation $\mathbb C$ arising as the space of all locally constant functions on $\Gamma\backslash G$.     
Since $G$ is connected, so is $\Gamma\backslash G$ and hence $m(\Gamma,\mathbb C)=1$.   The homomorphism $j^*:H^*(\mathfrak g,K;\mathbb C)\to H^*(\Gamma;\mathbb C)$ is referred to as the {\it Matsushima homomorphism}.

One has 
the {\it compact dual} $X_u$ of the globally symmetric space $X=G/K$ (where $G$ is any non-compact {\it real} semisimple Lie group).   The space $X_u$ is the compact Riemannian  symmetric space $U/K$, where $U$ is the maximal compact subgroup of the complexification of $G$  that 
contains $K$.  Its Lie algebra $\mathfrak u$ equals $\mathfrak k\oplus i\mathfrak p\subset \mathfrak g_\mathbb C=\mathfrak g \oplus i\mathfrak g$ where $\mathfrak g=\mathfrak k \oplus \mathfrak p$ is the Cartan decomposition corresponding to $K$.
It is well-known that $H^*(\mathfrak g, K;\mathbb C)\cong \textrm{Hom}_K(\Lambda ^*\mathfrak p,\mathbb C)\cong H^*(X_u;\mathbb C)$ (See \cite[\S3, Chapter II]{borel-wallach}).   

We shall now describe the compact dual $X_u$ when $G$ is a simply connected {\it complex} semisimple Lie group.  
  
Since $G$ is a complex Lie group regarded as a real Lie group, we have $\mathfrak p=i\mathfrak k$ and so 
 $\mathfrak u:=\mathfrak k \oplus i\mathfrak p\cong \mathfrak k\times \mathfrak k$.   
Since $\mathfrak u=\mathfrak k\times \mathfrak k$, it follows that $U\cong K\times K$.

Moreover, $K$ is imbedded diagonally in $U\cong K\times K$ and we have $U/K=K\times K/K\cong K$.    See \cite{helgason} for further details.    Therefore $X_u\cong K$ and 
we have $H^*(\mathfrak g,K;\mathbb C)=H^*(X_u;\mathbb C)\cong H^*(K;\mathbb C)$. The following result is an immediate consequence. 
\begin{proposition}
Let $G$ be a connected complex semisimple Lie group and $\Gamma$ be a torsionless uniform lattice in $G$. Then the Matsushima homomorphism 
defines an injective homomorphism $H^*(K;\mathbb C)\to H^*(\Gamma;\mathbb C)$.\hfill $\Box$
\end{proposition}

We have the following vanishing theorem.  

\begin{theorem} \label{vanishingtheorem} {\em (Cf. \cite[Corollary 4.4(b), Chapter VII] {borel-wallach})}
Let $\Gamma$ be an irreducible uniform lattice in a simply connected  complex semisimple Lie group $G$.  
Then the Matsushima homomorphism 
$j^*:H^q(\mathfrak g,K;\mathbb C)\to H^q(\Gamma;\mathbb C)$ is an isomorphism for $
1\le q< \rank(G)$.   \hfill $\Box$ 
\end{theorem}

A more refined version, applicable for connected real semisimple Lie groups with finite centre and without compact factors, is proved in \cite[Chapter VII, \S 4, Corollary 4.4(b)]{borel-wallach}.  
As an immediate consequence we obtain the following.

\begin{corollary}\label{cor-vanishing}  
If $\rank(G)\ge 2$, then $H^1(\Gamma;\mathbb C)=0$.  If $\rank(G)\ge 3$, then $H^2(\Gamma;\mathbb C)=0$. 
\end{corollary}
\begin{proof}
 We merely note that,  under our restrictive  hypotheses, 
$K$ is semisimple and simply connected. Moreover,  since the compact dual $X_u$ is diffeomorphic to $ K$, it follows that $H^q(\mathfrak g,K;\mathbb C)\cong H^q(X_u;\mathbb C)=0$ if $q=1,2$ as $K$ is semisimple.  So 
our assertion follows from Theorem \ref{vanishingtheorem}.

\end{proof}

\begin{remark}\label{bettinumber}    (i) 
The vanishing of $H^1(\Gamma;\mathbb R)$ when $\Gamma$ is an irreducible lattice in any connected real semisimple Lie 
group and  having rank at least $3$ was first proved by Kazhdan \cite{kazhdan}.  The result was extended to rank $2$ groups by 
S. P. Wang and by B.  Kostant.  

(ii) In the case when $\rank(G)=1$, $G=SL(2,\mathbb C)$and $X_\Gamma$ is a compact hyperbolic $3$-manifold with fundamental group 
$\Gamma$,  it is known that there are lattices $\Gamma$ such that $X_\Gamma$ has positive first Betti number.  
In fact, it was first shown by Millson \cite{millson} that  for any $n\ge 3$, there are uniform arithmetic lattices in $SO_0(1, n)$  whose first Betti numbers are arbitrarily large.  
\end{remark}

Let $M$ be a compact complex manifold.  As usual 
$H^q(M;\Omega^p_M)$ where $\Omega^p_M$ denotes the sheaf of holomorphic 
$p$-forms on $M$ will be denoted by $H^{p,q}(M)$.  Of course, $\Omega^0_M$ is the structure sheaf $\mathcal O_M$ of $M$.

{\it For the rest of this section, we assume that $\Gamma$ is a torsionless irreducible uniform lattice in 
a complex semisimple simply connected Lie group $G$.}

Since $G$ operates on (the right of) $M_\Gamma=\Gamma\backslash G$, it acts linearly on the finite dimensional complex vector space $H^{p,q}(M_\Gamma)$.   
The following theorem, which is in fact true for any connected reductive linear algebraic group $G$ over $\mathbb{C}$, which is due to D. N. Akhiezer, describes this $G$-representation.

\begin{theorem} \label{akhiezer} {\em(Akhiezer \cite{akhiezer2})}
  
Let $H^*(\Gamma;\mathbb C)$ be given the trivial $G$-module structure.  With the $G$-module structure on $\Lambda^*(\mathfrak g)$ arising from the adjoint action of $G$, we have an isomorphism 
\[ H^{p,q}(M_\Gamma)\cong H^q(\Gamma; \mathbb C)\otimes_\mathbb C\Lambda^p(\mathfrak g)\]
as $G$-modules for all $p,q\ge 0$. \hfill $\Box$

\end{theorem}

Denote by $\pi:M_\Gamma\to X_\Gamma$ the projection of the principal $K$-bundle.  
We need to compute the integral 
cohomology groups in degrees $q=1,2$ of $M_\Gamma$.  Since $K$ is semisimple and simply connected, we have 
$H^q(K;\mathbb Z)=0$ for $q=1,2$.  Also $H^1(X_\Gamma;\mathbb Z)=\hom(\Gamma,\mathbb Z)=0$ when the rank of $G$ is at least $2$ and 
moreover, $H^2(X_\Gamma;\mathbb Z)$ is a finite group when rank of $G$ is at least $3$ by Theorem \ref{vanishingtheorem}.  Therefore 
$H^2(X_\Gamma;\mathbb Z)=\textrm{Ext}(H_1(X_\Gamma;\mathbb Z),\mathbb Z)\cong \textrm{Ext}(\Gamma/[\Gamma,\Gamma],\mathbb Z)
\cong \Gamma/[\Gamma,\Gamma]$.  
 Applying the Serre spectral sequence to $\pi:M_\Gamma\to X_\Gamma$, we have $E_2^{p,q}=H^p(X_\Gamma;\mathcal H^q(K;\mathbb Z))=0$ when $q=1,2, p\ge 0$, and, $E_2^{1,0}=H^1(X_\Gamma;\mathbb Z)\cong \mathbb Z^r$ where $r=0$ when $\rank(G)\ge 2$.   It follows that 
$H^q(M_\Gamma; \mathbb Z)\cong H^q(X_\Gamma; \mathbb Z)$ when $q=1,2$.  The same conclusion holds when the coefficient group 
is $\mathbb C$.  Summarising, we have proved

\begin{proposition} \label{singularcohomologyofmgamma}
If $\rank(G)\ge 2$, then $H^1(M_\Gamma;\mathbb Z)=0$ and $H^2(M_\Gamma;\mathbb C)\cong H^2(\Gamma;\mathbb C)$.  If $\rank(G)\ge 3$, then $H^2(M_\Gamma;\mathbb Z)\cong 
\textrm{Ext}(H_1(M_\Gamma;\mathbb Z);\mathbb Z)$ which is isomorphic to the finite group $ \Gamma/[\Gamma,\Gamma]$. \hfill $\Box$  
\end{proposition}

\begin{remark}\label{simplecoeffts}
Let $G$ be a connected semisimple Lie group, not necessarily complex, and $\Gamma \subset G$ any torsionless 
lattice.  Consider the principal $K$-bundle with projection $M_\Gamma\to X_\Gamma$.  

{\em \underline{Claim:} The local coefficient system $\mathcal H^q(K;\mathbb Z)$ over $X_\Gamma=\Gamma\backslash G$ is simple, i.e., $\Gamma$ acts trivially on the cohomology of the fibre. }\\
This is true for the $K$-bundle $M_\Gamma\to X_\Gamma$ for any $G$ connected.  
Indeed let $\gamma\in \Gamma=\pi_1(X_\Gamma)$.  The element $\bar e:=\Gamma.K\in X_\Gamma$ is understood to be the base point.  We identify $k\in K$ with $\Gamma.k\in M_\Gamma$, the fibre over $\bar e$.  Choose a path $I\to G, t\mapsto \gamma_t$, that joins the identity element $e\in G$ to $\gamma\in \Gamma$.  
\[
\begin{array}{ccl}
0\times K&\stackrel{\Phi_0}{\longrightarrow} & M_\Gamma\\
\downarrow & \Phi\nearrow &\downarrow \pi\\
I\times K & \stackrel{\phi}{\to} & X_\Gamma\\
\end{array}
\]

Then $t\mapsto 
\Gamma\gamma_tK$ is a loop that represents $\gamma\in \pi_1(X_\Gamma)$.  The action of $\gamma\in \Gamma$ on 
$H^*(K;\mathbb C)$ is induced by $\Phi_1:K\cong 1\times K \to \pi^{-1}(\bar e)\cong K$ where 
$\Phi:I\times K\to M_\Gamma$ is a  lift of $\phi:I\times K\to 
X_\Gamma$ defined as $(t,k)\mapsto \Gamma \gamma_t K$ such that $\Phi(0,k)=\Gamma k~\forall k\in K$.   
Evidently $(t,k)\mapsto \Gamma \gamma_tk\in M_\Gamma$. We note that $\Phi_1(k)=\Gamma \gamma k=\Gamma k$ corresponds to the identity map of $K$. This proves our claim.
\end{remark}

We shall now prove Theorem \ref{picardgroup}.

{\it Proof of Theorem \ref{picardgroup}:} (i) Assume that the rank of $G$ is at least $2$.
Consider the long exact sequence induced by the exponential sequence of sheaves, where we have replaced 
$H^1(M_\Gamma;\mathcal O^*_{M_\Gamma})$ by $Pic(M_\Gamma)$:
\begin{equation}
\to H^1(M_\Gamma;\mathbb Z)\stackrel{i_*}{\to} H^1(M_\Gamma;\mathcal O_{M_\Gamma})\stackrel{\varepsilon_*}{\to} \Pic(M_\Gamma)\stackrel{c_1}{\to} H^2(M_\Gamma;\mathbb Z)\to H^2(M_\Gamma;\mathcal O_{M_\Gamma})\to \cdots
\end{equation}
where $c_1$ is the Chern class map.  

We have $H^1(M_\Gamma;\mathbb Z)=0=H^1(M_\Gamma;\mathcal O_{M_\Gamma})$ and so it follows that $Pic^0(M_\Gamma)=0$.  

When $\rank (G)\ge 3$, we have $H^2(M_\Gamma;\mathcal{O}_{M_{\Gamma}})\cong H^2(\Gamma;\mathbb C)\cong 0$ where the first isomorphism is by Theorem \ref{akhiezer} and  the second by Theorem \ref{vanishingtheorem}. Proposition \ref{singularcohomologyofmgamma} implies
$H^2(M_\Gamma;\mathbb Z)\cong \Gamma/[\Gamma,\Gamma]$ is a finite group and so our assertion follows. 

In the rank $1$ case, namely when $G=SL(2,\mathbb C),$ we have $K=SU(2)\cong \mathbb S^3$ and $X$ is the three-dimensional real hyperbolic space.    Let $H^1(\Gamma;\mathbb Z)\cong \mathbb Z^r$.  By the Poincar\'e duality,  we have $H^2(\Gamma;\mathbb Z)\cong \mathbb Z^r\oplus A$ where $A$ is isomorphic to the torsion 
subgroup of $H_1(\Gamma;\mathbb Z)=\Gamma/[\Gamma,\Gamma]$.   Then $H^1(M_\Gamma;\mathcal O_{M_\Gamma})\cong \mathbb C^r, 
H^2(M_\Gamma;\mathcal O_{M_\Gamma})\cong 
H^2(\Gamma;\mathbb C)\cong \mathbb C^r$, by 
Theorem \ref{akhiezer}. Using the Serre spectral sequence we get that $H^q(M_\Gamma;\mathbb Z)\cong H^q(\Gamma;\mathbb Z), q=1,2.$ 
It follows from (2) that $\Pic^0(M_\Gamma)\cong \mathbb C^r/\Lambda$ where $\Lambda=H^1(M_\Gamma;\mathbb Z)\cong \mathbb Z^r$. 
\hfill $\Box$

When $\rank(G)\leq 2$, in order to determine $Pic(M_\Gamma)$ it remains to compute the 
image of the Chern class map $H^1(M_\Gamma;\mathcal {O}_{M_{\Gamma}}^{*})\to H^2(M_\Gamma;\mathbb Z).$   Equivalently, 
we need only to determine the kernel of $H^2(M_\Gamma;\mathbb Z) \to H^2(M_\Gamma,\mathcal{O}_{M_{\Gamma}})$.   
We have the following diagram
\[
\begin{array}{ccc}  
H^2(\Gamma;\mathbb C)&\stackrel{\psi}{\to}& H^2(M_\Gamma;\mathcal O_{M_\Gamma})\\

\pi^* \downarrow 
&&\downarrow id\\

H^2(M_\Gamma;\mathbb C)&\stackrel{\iota_*}\to &H^2(M_\Gamma;\mathcal O_{M_\Gamma})\\
\end{array}
\]
in which $\psi$ is the Akhiezer isomorphism of Theorem \ref{akhiezer}, $\iota:\underline{\mathbb C}\to \mathcal O_{M_\Gamma}$ is the inclusion 
where $\underline{\mathbb C}$ denotes the constant sheaf over $X$.  The vertical map $\pi^*$ is induced by 
the projection of the $SU(2)$-principal bundle, which is an isomorphism (in degree $2$).  It is plausible 
that the above diagram commutes and so $\iota_*$ is an isomorphism, but it appears to be hard to establish. 

When $\rank(G)=2$, we have $H^{1,1}(M_\Gamma)=0$ by Theorem \ref{akhiezer}.

A result of Dolbeault \cite[Th\'eor\`eme 2.3]{dolbeault} establishes the Lefschetz theorem on $(1,1)$-classes for {\it non-K\"ahler} manifolds. 
Accordingly, an element $c\in H^2(M_\Gamma;\mathbb Z)$ is the first Chern class of a holomorphic line bundle over $M_\Gamma$ if and only if it is represented by a 
form of type $(1,1)$, say $\omega.$  
However, as $M_\Gamma$ is not K\"ahler, it does not seem to follow that $c$, regarded as an element of $H^2(M_{\Gamma};\mathbb R)$, vanishes although $[\omega]\in H^{1,1}(M_\Gamma)=0$.

\begin{remark} \label{remarksonmaintheorem1}
(i) 
When $\Gamma\subset G$ is as in Theorem \ref{picardgroup} and rank of $G$ is at least two, the abelianisation $\Gamma/[\Gamma,\Gamma]\cong H_1(M_\Gamma;\mathbb Z)$ is finite.   The torsion subgroup $A$ of $H^2(M_\Gamma;\mathbb Z)$ is isomorphic to 
$\textrm{Ext}(H_1(M_\Gamma;\mathbb Z),\mathbb Z) 
\cong \textrm{Ext}(\Gamma/[\Gamma,\Gamma],\mathbb Z)\cong \Gamma/[\Gamma,\Gamma]$ since the abelianisation of $\Gamma$ is finite.  In particular,  
$A$ is zero when $\Gamma$ is perfect. Moreover if $\rank(G)\geq 3$ and $\Gamma$ is perfect, then $Pic^{0}(M_\Gamma)=Pic(M_{\Gamma})=0$.

(ii) 
If $\chi : \Gamma\to \mathbb C^*$ is any homomorphism, we obtain a holomorphic line bundle $L_\chi$ with total space $G\times_\Gamma\mathbb C$.   When $\rank(G)\ge 2$, the image of $\chi$ is necessarily a finite cyclic subgroup $\mu_m\subset \mathbb C^*$.  The homomorphism $\chi_m: \Gamma/[\Gamma,\Gamma]\to \mu_m$ defined by $\chi$ determines an element $\alpha$ of $H^2(\Gamma;\mathbb Z)=\textrm{Ext}(\Gamma,\mathbb Z)$ corresponding to the image of the generator $\textrm{Ext}(\mu_m;\mathbb Z)\cong \mathbb \mu_m$ under 
$\chi^*_m: H^2(\mu_m;\mathbb Z)=\textrm{Ext}(\mu_m;\mathbb Z)\to \textrm{Ext}(\Gamma/[\Gamma,\Gamma],\mathbb Z)\hookrightarrow H^2(\Gamma;\mathbb Z)$.    Then,  under the isomorphism  $ \pi^{*} : H^{2}(\Gamma;\mathbb{Z})\cong H^{2}(M_{\Gamma};\mathbb{Z})$, (via the Serre spectral sequence)  $c_1(L_\chi)=\pi^{*}(\alpha)$.

(iii) The computation of Betti numbers $b_j$ of $M_\Gamma=\Gamma\backslash SL(2, \mathbb C)$ for a torsionless uniform lattice can be 
completed using the Serre spectral sequence  $E^{p,q}_r$ of the principal $SU(2)$-bundle over the hyperbolic manifold $
X_\Gamma$.  As we saw in the above proof,  $b_1=b_2=\rank( \Gamma/[\Gamma,\Gamma])$.  Since the group local coefficient system $\mathcal H^3(SU(2);\mathbb C)$ over $X_\Gamma$ is trivial (by 
Remark \ref{simplecoeffts}),  we have $E_2^{0,3}=E_\infty^{0,3}\cong H^3(SU(2);\mathbb C)\cong \mathbb C$. Moreover, $E_\infty^{3,0}=E_2^{3,0}=H^3(\Gamma;\mathbb C)=\mathbb C$
as $X_\Gamma$ is a compact connected $3$-manifold.  So $b_3=2$. By Poincar\'e duality (applied to $M_\Gamma$)  we have 
$b_4=b_2=r=b_1=b_5$.  Of course $b_0=b_6=1, b_j=0 $ for $j>6$.  This was already observed by Akhiezer \cite{akhiezer2}.  
\end{remark}

\section{Picard groups of $M_\Gamma$-bundles}

In this section we prove the main result of this paper, namely, Theorem \ref{bundleversionofPic}. 

Suppose that $Y$ is a compact connected complex manifold.  Let $P\to Y$ be a holomorphic principal $G$-bundle where $G$ is a simply connected complex linear semisimple Lie group.   We assume that $G$ acts on the left of $P$. 
 Let $\Gamma$ be a torsionless irreducible and uniform lattice in $G$.  Let $P_\Gamma=\Gamma \backslash P$.   Then we have the natural projection 
$P_\Gamma\to Y$ of a holomorphic $M_\Gamma=\Gamma\backslash G$-bundle over $Y$ with structure group $G$. 
The proof will involve computing $H^q(P_\Gamma;\mathcal O_{P_\Gamma})$ for $q=1,2$ 
using the Borel spectral sequence \cite[Appendix II, \S2]{hirzebruch}.   
\subsection{Borel spectral sequence}\label{borelss}
The hypotheses of the Borel spectral 
sequence include the requirement that the connected components of the 
structure group of the holomorphic bundle act trivially on the $\bar\partial$-cohomology of the fibre.  This is 
(in general) not true for $M_\Gamma$-bundles.  However, one can still use the spectral sequence 
for the computation of $H^*(P_\Gamma;\mathcal O_{P_{\Gamma}})$, as we shall now explain.

More generally, let 
$\xi=(E,B,F,\pi)$ be a holomorphic fibre bundle with structure group $G$, which is a complex Lie group. Suppose that the fibre space $F$ is compact and that 
$E,B,F$ are connected.  Assume that the connected components of $G$ act trivially on $H^*(F;\mathcal O_F)$.  
Then one has a {\it holomorphic} vector bundle $\mathcal H^q(F;\mathcal O_F)$ over $B$ whose fibre over 
$b\in B$ is the complex vector space 
$H^{q}(F_b;\mathcal O_{F_b})\cong H^q(F;\mathcal O_F)$ where $F_b=\pi^{-1}(b)$.  
 Let $W$ be a holomorphic vector bundle over $B$ and let $\hat W:=\pi^*(W)$. 
Then one has a spectral sequence 
${}^{0,q}E_r^{s,t}={}^qE_r^{s,t}$ in which the differential $d_r:{}^qE_r^{s,t}\to {}^{q+1}E_r^{s+r,t-r+1}$ has bidegree $(r,1-r)$.   We have 
${}^qE^{s,t}_r=0$ unless $t=q-s$;
$s$ is the `base degree', $t$ is the `fibre degree'.   The $E_2$-page is given as ${}^qE_2^{s,t}=H^s(B; W\otimes \mathcal{H}^{t}(F;\mathcal O_F))$.  The spectral sequence converges to $H^{0,q}(E;\hat{W})=H^q(E;\mathcal O(\hat{W}))$.  
That is, for each $q\ge 0$, there exists a filtraion of $H^q(E;\hat{W})$ such that the associated graded space is:
\[ Gr H^q(E;\mathcal O(\hat{W}))=\sum_{0\le s\le q} {}^qE_\infty^{s,q-s}.\]

The proof is exactly as given in \cite[Appendix II]{hirzebruch}, where we need only replace the filtration $L_k$ (in \S4 therein) by $L_k\cap A^{0,*}_E(\hat W)$.   For the sake 
of completeness we indicate below the crucial place where the change is necessitated.   

Define 
$L_k(U)\subset A_U(\hat{ W}|U)$, for a small open set $U\subset E$, to be the span 
of monomials $d\bar z_J\wedge d\bar y_{J'}$ in which $|J|+|J'|\ge k$.  Thus $L_k(U)\subset \sum_{q\ge k}A^{0,q}(U)$.  
Here, the smallness refers to both $\xi$ and $W$ being trivial over $\pi(U)=:V$ and $U\cong V\times V'$ (via a local  analytic chart of $\xi$),  where $(V,z_i), (V',y_j)$ are holomorphic coordinate charts in $B,F$ respectively. 
Define 
\[L_k:=\{\omega\in A_E(\hat{ W})\mid \omega|_U\in L_k(U) ~\forall \textrm{~small open subset~}U\subset E\}.\]
Then 
\[L_0=A_E(\hat{W}), L_k=0 ~\textrm{~if~} k>\dim_\mathbb R B; L_k\supset L_{k+1}; \bar\partial L_k\subset L_k~\forall k\ge 0. \]
Also  $L_k=\sum_{q\ge 0} {}^{0,q}L_k$ where ${}^{0,q}L_k=L_k\cap A_E^{0,q}(\hat{W})$. 
Note that $H^*(E;\mathcal O_E(\hat W))$ is the cohomology of the cochain complex $(A^{0,*}_E(\hat{W}),\bar
\partial)$.  
The required spectral sequence is associated to the filtration $\{(L_k,\bar\partial)\}$ of the 
differential graded complex $(A^{0,q}_E(\hat{ W}),\bar\partial)$.     The proof 
is exactly as given by Borel in \cite{hirzebruch}.
\subsection{Proof of Theorem \ref{bundleversionofPic}}
Reverting back to the $M_\Gamma$-bundle with projection $P_\Gamma\to Y$, 
first, observe that the structure group 
$G$ of the $M_\Gamma$-bundle with projection $P_\Gamma\to Y$ acts trivially 
on $H^q(M_\Gamma;\mathcal O_{M_\Gamma})\cong H^q(X_\Gamma;\mathbb C)$ by Theorem \ref{akhiezer}.   So 
the Borel spectral sequence 
is applicable for the $M_\Gamma$-bundle for computing $H^q(P_\Gamma;\mathcal O_{P_\Gamma})$.   
The $E_2$-page of the spectral sequence (with $W$ being the trivial bundle)
is given by:  ${}^{0,q}E_2^{s,t}=H^{s}(Y;\mathcal H^{0, q-s}(M_\Gamma))$.   Here $\mathcal H^{0,q-s}(M_\Gamma)$ denotes the holomorphic vector bundle over $Y$ with discrete structure group.  Since 
$G$ is connected, the structure group of the vector bundle $\mathcal H^{0, t}(M_\Gamma)$ reduces to the trivial group and so 
\begin{equation}
{}^{0,q}E^{s,t}_2=H^s(Y;\mathcal O_Y)\otimes H^t(M_\Gamma;\mathcal O_{M_\Gamma})
\end{equation}   
We shall suppress the type $(0,q)$ in the notation ${}^{0,q}E_r^{s,t}$.  
Our main interest is in computing $H^q(P_\Gamma;\mathcal O_{P_\Gamma})$ when $q=1,2$.   

We will assume that $\rank(G) \ge 2$, so that $H^{0,1}(M_\Gamma)=0$ by Theorems \ref{vanishingtheorem} and \ref{akhiezer}.  
Substituting in (3), we obtain 
$E_2^{s,1}=0~\forall s\ge 0$.  We have $E_2^{0,q}\cong 
H^q(M_\Gamma;\mathcal O_{M_\Gamma}),q\ge 2;  E_2^{s,0}\cong H^s(Y; \mathcal O_Y)~\forall s\ge 0$. 
  
It follows that $H^1(P_\Gamma;\mathcal O_{P_\Gamma}) \cong H^1(Y;\mathcal O_Y)$, the isomorphism being induced by $\pi$.    Also, 
 $H^2(P_\Gamma;\mathcal O_{P_\Gamma})\cong H^2(Y;\mathcal O_Y)\oplus E_\infty ^{0,2}=H^2(Y;\mathcal O_Y)\oplus E_4^{0,2}$.    We note that 
$E_4^{0,2}=E_2^{0,2}=H^2(M_\Gamma;\mathcal O_{M_\Gamma})=0$ if $\rank(G)>2$.  Summarising we have 

\begin{lemma} \label{h2pgamma}
We keep the above notations. Suppose that $\rank(G)\ge 2$.  Then $H^1(P_\Gamma;\mathcal O_{P_\Gamma})\cong H^1(Y;\mathcal O_Y)$ and 
$H^2(P_\Gamma;\mathcal O_{P_\Gamma})\cong H^2(Y;\mathcal O_{Y})\oplus V$ for some suitable vector subspace $V\subset 
H^2(M_\Gamma;\mathcal O_ {M_\Gamma})$. If $\rank(G)\ge 3$, then $V=0$. \hfill $\Box$
\end{lemma}

We now turn to the computation of $H^q(P_\Gamma;\mathbb Z)$ for $q=1,2$.   Using the homotopy exact sequence of the $M_\Gamma$-bundle 
with projection $\pi: P_\Gamma\to Y$, we see that 
$\pi_*: \pi_1(P_\Gamma) \to \pi_1(Y)$ is surjective.  Therefore $H^1(Y;\mathbb Z)\to H^1(P_\Gamma;\mathbb Z)$ is injective.   We claim 
that it is an isomorphism.  
Using the Serre spectral sequence 
and the fact $H^1(M_\Gamma;\mathbb Z)=0$ (as $H_1(M_\Gamma;\mathbb Z)=\Gamma/[\Gamma,\Gamma]$ is finite), we see that $\pi^*:H^1(Y;\mathbb Z)\to H^1(P_\Gamma;\mathbb Z)$  is an isomorphism.

Next we turn to computation of $H^2(P_\Gamma;\mathbb Z)$.  
Assume that $\rank(G)\ge 3$ so that $H^2(\Gamma;\mathbb Z)$ is finite by Theorem \ref{vanishingtheorem}.  Since $G$ is a {\it complex} semisimple  simply connected Lie group, $K$ is also semisimple and simply connected.   Hence $H^q(K;\mathbb Z)=0$ for $q=1,2$.
 Using the Serre spectral sequence for the principal $K$-bundle with projection $M_\Gamma\to X_\Gamma$, we obtain that $H^2(M_\Gamma;\mathbb Z)\cong H^2(\Gamma;\mathbb Z)$ is finite. 
Hence, in the Serre spectral sequence for the $M_\Gamma$-bundle, we have $E^{0,2}_2=H^0(Y;\mathcal H^2(M_\Gamma;\mathbb Z))$ is finite.  
This, together with the vanishing of $H^1(M_\Gamma;\mathbb Z)$, implies that  
that $\pi^*:H^2(Y;\mathbb Z)\to H^2(P_\Gamma;\mathbb Z)$ is injective and has finite cokernel.

We shall now prove Theorem \ref{bundleversionofPic}.

{\it Proof of Theorem \ref{bundleversionofPic}.}  Suppose that $\rank(G)\ge 2$.   We have a commuting diagram induced by the exponential exact sequence, where 
$\pi:P_\Gamma\to Y$ is the projection of the $M_\Gamma$-bundle.  
\begin{equation}
\begin{array}{ccccccccc}
 H^1(Y;\mathbb Z)&\to& H^1(Y;{\mathcal {O}}_{Y})&\to &\Pic(Y)&\stackrel{c_1}{\to} &H^2(Y;\mathbb Z)&\to&H^2(Y;{\mathcal {O}}_{Y}) \\
\pi^* \downarrow &                    &\pi^* \downarrow &&\pi^* \downarrow                               && \pi^* \downarrow &&\downarrow \pi^*\\
 H^1(P_\Gamma;\mathbb Z)&\to & H^1(P_\Gamma;{\mathcal {O}}_{P_{\Gamma}})&\to & \Pic(P_\Gamma)&\stackrel{c_1}{\to} &H^2(P_\Gamma;\mathbb Z)&\to &H^2(P_\Gamma;{\mathcal {O}}_{P_{\Gamma}})\\
\end{array}
\end{equation}
As noted above, $\pi^*:H^1(Y;\mathbb Z)\to H^1(P_\Gamma;\mathbb Z)$ is an isomorphism.  
By Lemma \ref{h2pgamma}, $\pi^*:H^q(Y;\mathcal O_Y)\to H^q(P_\Gamma;\mathcal O_{P_\Gamma})$ is an isomorphism when $q=1$ as $
\rank(G)\ge 2$.  Denoting the inclusion of sheaves $\mathbb Z\hookrightarrow \mathcal O_Y$ by $i_Y$ and a similar notation for $P_\Gamma$, it follows that the natural homomorphism 
$\Pic^0(Y)=H^1(Y;\mathcal O_Y)/Im(i_{Y,*})\to H^1(P_\Gamma;\mathcal O_{P_\Gamma})/Im(i_{P_\Gamma,*})
\cong \Pic^0(P_\Gamma)$ is surjective.  The commutativity of the left-most square in (4) and the fact that the $\pi^*$
in that square are isomorphisms implies that $Pic^0(Y)\to Pic^0(P_\Gamma)$ is an isomorphism.  

Now let $\rank(G)\ge 3$.  It remains to show that $\pi^*(Pic(Y))\subset Pic(P_\Gamma)$ has finite index.   Suppose that $L$ is a holomorphic line bundle 
over $P_\Gamma$.  If $c_1(L)=0$, then $L\in Pic^0(P_\Gamma)$ and so $L=\pi^*(L')$ for some line bundle $L'$ over $Y$ as $\pi^*:Pic^0(Y)\to Pic^0(P_\Gamma)$ is an isomorphism.     
Now suppose that $c_1(L)=a\ne 0$.  Choose $m\ge 1$ to be 
the cardinality of the cokernel of $\pi^*: H^2(Y;\mathbb Z)\to H^2(P_\Gamma;\mathbb Z)$.   It suffices to show that $L^m$ is in the image of 
$\pi^*:Pic(Y)\to Pic(P_\Gamma)$.  
We have $ma=\pi^*(b)$ for some $b\in H^2(Y;\mathbb Z)$.   By the commutativity of the right-most square in (4) and the observation 
that $\pi^*:H^2(Y;\mathcal O_Y)\to H^2(P_\Gamma;\mathcal O_{P_\Gamma})$ is an isomorphism (by Lemma \ref{h2pgamma}), 
we see that $b$ is in the kernel of $H^2(Y;\mathbb Z)\to H^2(Y;\mathcal O_Y)$.  Therefore 
$b=c_1(L_1)$ for some line bundle $L_1$ over $Y$.  Hence $L^{-m}\pi^*(L_1)\in Pic^0(P_\Gamma)$.   Choose $L_0\in Pic^0(Y)$ such that $L^{-m}\pi^*(L_1)
=\pi^*(L_0)$. Then $L^m=\pi^*(L_0^{-1}L_1)$.  
   \hfill $\Box$
\begin{example}  {\em 
 Let $\omega$ be any vector bundle of rank $k$ over a compact connected complex manifold $Y$.  Let $\xi=\Lambda^k(\omega)^*$.  
Then we claim that $\eta:=\omega\oplus \xi$ admits a reduction of the structure group to $SL(k+1,\mathbb C)$.  This is because 
$\Lambda^{k+1}(\eta)=\Lambda^{k+1}(\omega\oplus \xi)=\Lambda^k(\omega) \otimes \xi=\varepsilon$, the trivial line bundle.   We take 
$P\to Y$ to be the associated principal  $SL(k+1,\mathbb C)$-bundle.  For any uniform irreducible torsionless lattice we obtain a holomorphic 
bundle $P_\Gamma\to Y$ with fibre $SL(k+1,\mathbb C)/\Gamma. $}
  \end{example}
\begin{remark} Let $\rank(G)\ge 3$.  Suppose that $Y$ (as in Theorem \ref{bundleversionofPic}) is simply connected. Then $\pi_1(M_\Gamma)\cong\Gamma\to \pi_1(P_\Gamma)$ is a surjection.  Also $\pi^*: H^2(Y;\mathbb Z)\to H^2(P_\Gamma;\mathbb Z)$ is a monomorphism as can be seen using the Serre spectral sequence. 
Suppose that $H^2(Y;\mathbb Z)\to H^2(Y;\mathcal O_Y)$ is the trivial homomorphism so that $c_1:Pic(Y)\to H^2(Y;\mathbb Z)$ is surjective.  (For example, take $Y$ such that $H^2(Y;\mathcal O_Y)=0$.)  Let $C$ denote the subgroup of $Pic(M_\Gamma)$ defined as the image of 
the restriction to a fibre $Pic(P_\Gamma)\to Pic(M_\Gamma)$.  Since $\rank(G)\ge 3$, we have $Pic(M_\Gamma)\cong \Gamma/[\Gamma,\Gamma]=A$, which is a finite abelian group.  Finally, let $C_1$ denote the image of $c_1:Pic(P_\Gamma)\to H^2(P_\Gamma;\mathbb Z)$. 

Then we have 
the following commuting diagram
\[\begin{array}{ccccc}
Pic^0(Y)&\stackrel{\cong}{\to} & Pic^0(P_\Gamma)&&\\
\downarrow && \downarrow &&\\
Pic(Y)&\to &Pic(P_\Gamma)& \to &Pic(M_\Gamma)\\
\downarrow && \downarrow &&\downarrow  \cong\\
H^2(Y;\mathbb Z)&\to &C_1&\to& A\\
\end{array}\]
An easy diagram chase reveals that we have an exact sequence \[1\to Pic(Y)\to Pic(P_\Gamma)\to C\to 1.\]   
If $\Gamma$ is perfect, then $Pic(M_\Gamma)=0$ and we have $Pic(Y)\cong Pic(P_\Gamma)$.  
\end{remark}
\noindent
{\bf Acknowledgments}  We thank V. Balaji and D. S. Nagaraj for their valuable comments. 
We thank the referee for his/her comments which helped us to improve the exposition. 

\end{document}